\documentclass[a4paper,11pt]{amsart}  
\usepackage[a4paper,hmargin={2.5cm,2.5cm},vmargin={2.5cm,2.5cm}]{geometry}

\usepackage[square]{natbib} 
\usepackage[leqno]{amsmath}
\usepackage{amssymb,amsthm}
\usepackage[mathscr]{euscript}
\usepackage{stmaryrd}
\usepackage{bbm}
\usepackage{dsfont}
\usepackage[all,arc]{xy}
\usepackage{longtable}
\DeclareMathAlphabet{\mathpzc}{OT1}{pzc}{m}{it}

\theoremstyle{plain}
\newtheorem{theorem}{Theorem}[section]
\newtheorem{lemma}[theorem]{Lemma}
\newtheorem{proposition}[theorem]{Proposition}
\newtheorem{corollary}[theorem]{Corollary}

\theoremstyle{definition}
\newtheorem{definition}[theorem]{Definition}
\newtheorem{examples}[theorem]{Examples}

\theoremstyle{remark}
\newtheorem{remark}[theorem]{Remark}

\newcommand{\Rw}{\Rightarrow}

\newcommand{\Lw}{\Leftarrow}

\newcommand{\hrw}{\hookrightarrow}

\newcommand{\frf}{\mathfrak{f}}
\newcommand{\frg}{\mathfrak{g}}

\newcommand{\frx}{\mathfrak{x}}

\newcommand{\calA}{\mathcal{A}}

\newcommand{\calO}{\mathcal{O}}

\newcommand{\frF}{\mathfrak{F}}

\newcommand{\frX}{\mathfrak{X}}

\DeclareMathOperator{\yoneda}{\mathpzc{y}}

\DeclareMathOperator{\yonmult}{\mathpzc{m}}
\DeclareMathOperator{\Sup}{Sup}
\DeclareMathOperator{\upc}{\uparrow\!}
\DeclareMathOperator{\downc}{\downarrow\!}
\DeclareMathOperator{\Clos}{Cl}
\DeclareMathOperator{\kar}{kar}
\DeclareMathOperator{\Up}{Up}
\DeclareMathOperator{\tbelow}{\ll_{_\mT}}
\DeclareMathOperator{\und}{\&}

\newcommand{\catfont}[1]{\mathsf{#1}}

\newcommand{\SET}{\catfont{Set}}
\newcommand{\REL}{\catfont{Rel}}
\newcommand{\ORD}{\catfont{Ord}}
\newcommand{\MOD}{\catfont{Mod}}
\newcommand{\DLAT}{\catfont{DLat}}
\newcommand{\SLAT}{\catfont{SLat}}
\newcommand{\TOP}{\catfont{Top}}

\newcommand{\FRM}{\catfont{Frm}}
\newcommand{\Frm}[1]{\catfont{Frm}_{\wedge,#1}}
\newcommand{\SLat}[1]{\catfont{SLat}_{\wedge,#1}}
\newcommand{\Spl}[1]{\catfont{Spl}(#1)}

\newcommand{\two}{\catfont{2}}

\newcommand{\relto}{{\longrightarrow\hspace*{-2.8ex}{\mapstochar}\hspace*{2.6ex}}}

\newcommand{\monadfont}[1]{\mathbbm{#1}}

\newcommand{\mT}{\monadfont{T}}
\newcommand{\mF}{\monadfont{F}}
\newcommand{\mD}{\monadfont{D}}
\newcommand{\mFa}[1]{\monadfont{F}_{_{\! #1}}}

\newcommand{\monad}{(T,\yoneda,\yonmult)}
\newcommand{\fmonad}{(F,\yoneda,\yonmult)}
\newcommand{\dmonad}{(D,\yoneda,\yonmult)}

\newcommand{\fmonada}[1]{(F_{_{\! #1}},\yoneda,\yonmult)}

\newcommand{\doo}[1]{\overset{\centerdot}{#1}}
\newcommand{\eps}{\varepsilon}
\newcommand{\op}{\mathrm{op}}

\newcommand{\field}[1]{\mathds{#1}}

\newcommand{\N}{\field{N}}

\newcommand{\R}{\field{R}}

\begin{document}

\title{A four for the price of one duality principle for distributive spaces}

\author{Dirk Hofmann}
\thanks{Partial financial assistance by Centro de Investiga\c{c}\~ao e Desenvolvimento em Matem\'atica e Aplica\c{c}\~oes da Universidade de Aveiro/FCT and the project MONDRIAN (under the contract PTDC/EIA-CCO/108302/2008) is gratefully acknowledged.}
\address{Departamento de Matem\'{a}tica\\ Universidade de Aveiro\\3810-193 Aveiro\\ Portugal}
\email{dirk@ua.pt}
\begin{abstract}
In this paper we consider topological spaces as generalised orders and characterise those spaces which satisfy a (suitably defined) topological distributive law. Furthermore, we show that the category of these spaces is dually equivalent to a certain category of frames by simply observing that both sides represent the idempotents split completion of the same category.
\end{abstract}

\date{} 

\maketitle

\section*{Introduction}

Our work with topological spaces presented as convergence structures shaped the idea that \emph{topological spaces are generalised orders}, and therefore can be studied using notions and techniques from order theory. To start with, the reflexivity and the transitivity axiom
\begin{align*}
x&\le x &\text{and}&& x\le y\,\und\,y\le z\,&\Rw\,x\le z
\intertext{of and ordered set match exactly the conditions (see \citep{Bar_RelAlg})}
\doo{x}&\to x &\text{and}&& \frX\to\frx\,\und\,\frx\to x\,&\Rw\,m_X(\frX)\to x
\end{align*}
characterising those convergence relations between ultrafilters and points coming from a topology; and further work revealed also the following analogies.
\begin{longtable}{|c|c|}
\hline
\hspace*{2em}For an ordered set $X$:\hspace*{2em} & \hspace*{2em}For a topological space $X$:\hspace*{2em}\\
\hline
up-closed subset & closed subset\\
down-closed subset & filter of opens\\
non-empty down-closed subset & proper filter of opens\\
directed down-closed subset & prime filter of opens\\
upper bound & limit point\\
supremum & smallest limit point\\
cocomplete\footnote{We use the term ``\emph{co}complete'' here since we require the existence of suprema, which are \emph{co}limits.} ordered set & continuous lattice\\
directed cocomplete ordered set & stably compact space\\
down-set monad on $\ORD$ & filter monad on $\TOP$\\
\hline
\end{longtable}

Another concept of order theory which offers itself in this context is that of distributivity (or continuity). Recall that a cocomplete ordered set $X$ is completely distributive whenever arbitrary infima commute with arbitrary suprema in $X$. An elegant way to express this was found in \citep{FW_CCD}, where it is shown that $X$ is completely distributive if and only if the monotone map defined by $A\mapsto\Sup A$ ($A\subseteq X$ is down-closed) has a left adjoint. Similarly, a directed cocomplete ordered set $X$ is continuous if and only if the map $A\mapsto\Sup A$ ($A\subseteq X$ is directed and down-closed) has a left adjoint.

Motivated by the latter formulations, we say that a cocomplete topological space ($=$ continuous lattice) is \emph{completely distributive} if the continuous map sending a filter of opens $\frf$ to its smallest limit point has a left adjoint in $\TOP$. Fortunately, not only the notion of but also several results about distributivity can indeed be imported to topological spaces (see \cite{HW_AppVCat,Hof_DualityDistSp}); however, so far we were not able to formulate the concept of distributive space in elementary topological terms. It is the aim of this note to improve this situation, and in Theorem \ref{thm:Tdist} we show that distributivity relates to a certain disconnectedness condition.

Finally, in the last section we will show that the category of distributive spaces is dually equivalent to a certain category of frames by simply observing that both sides represent the idempotents split completion of the same category.

\section{Filter monads on $\TOP$}

For a topological space $X$, we write $\calO X$ for the collection of open subsets of $X$, and $\calO(x)$ for the set of all open neighbourhoods of $x\in X$. The category of topological spaces and continuous maps we denote as $\TOP$. We emphasise that $\TOP$ is an \emph{ordered} category, meaning there is a (non-trivial) order on the set $\TOP(X,Y)$ of all continuous maps between $X$ and $Y$ which is preserved by composition from either side. In fact, the topology of a space $X$ induces an order relation on the set $X$: for $x,x'\in X$ we put
\begin{align*}
x \le x' &:\iff\doo{x}\to x'\iff \calO(x')\subseteq\calO(x).
\end{align*}
Note that our order is \emph{dual} to the \emph{specialisation order}.
This relation is always reflexive and transitive, it is anti-symmetric if and only if $X$ is T$_{_0}$. The lack of anti-symmetry is not really problematic; however, it usually forces us to substitute equality $x=y$ by equivalence $x\simeq y$, where $x\simeq y$ whenever $x\le y$ and $y\le x$. The underlying order of topological spaces can be transferred point-wise to continuous maps $f,g:X\to Y$: $f\le g$ whenever $f(x)\le g(x)$ for all $x\in X$. 

The ordered character of $\TOP$ allows us to speak about \emph{adjunction}: continuous maps $f:X\to Y$ and $g:Y\to X$ form an adjunction, written as $f\dashv g$, if $1_X\le g\cdot f$ and $f\cdot g\le 1_Y$. One calls $f$ left adjoint and $g$ right adjoint, motivated by the fact that $f\dashv g$ if and only if
\[
 Uf(\frx)\to y\iff \frx\to g(y),
\]
for all ultrafilter $\frx\in UX$ and all points $y\in Y$\footnote{If we would have chosen the specialisation order and not its dual, then the left adjoint would appear on the right hand side in the formula above.}. We recall that adjoints determined each other, that is, $f\dashv g$ and $f\dashv g'$ imply $g\simeq g'$, likewise, $f\dashv g$ and $f'\dashv g$ imply $f\simeq f'$.

The \emph{filter monad} $\mF=\fmonad$ on $\TOP$ as well as many of its submonads are extensively described in \citep{Esc_InjSp,EF_SemDom} (and we refer to \citep{MS_Monads} for more information on monads). We recall that the filter functor $F:\TOP\to\TOP$ sends a topological space $X$ to the space $FX$ of all filters on the lattice $\calO X$ of open subsets of $X$, and the topology on $FX$ is generated by the sets
\begin{align*}
A^\#&=\{\frf\in FX\mid A\in\frf\}
\end{align*}
where $A\subseteq X$ is open. The induced order relation on $FX$ is dual to the inclusion order, that is, $\frf\le\frg$ whenever $\frf\supseteq\frg$. For $f:X\to Y$ continuous, $Ff:FX\to FY$ is defined by
\[
 \frf\mapsto\{B\subseteq Y\mid f^{-1}(B)\in\frf\},
\]
for $\frf\in FX$. Since $Ff^{-1}(B^\#)=(f^{-1}(B))^\#$ for every $B\subseteq Y$ open, $Ff$ is indeed continuous. The unit $\yoneda_X:X\to FX$ sends a point $x\in X$ to its neighbourhood filter $\calO(x)$, and $\yonmult_X:FFX\to FX$ sends $\frF\in FFX$ to the filter $\{A\subseteq X\mid A^\#\in\frF\}$.

In this paper we think of the filter monad as the topological version of the ``down-set monad''. To explain this analogy, we recall that the down-set monad $\mD=\dmonad$ on $\ORD$ is defined by the down-set functor $D:\ORD\to\ORD$ which sends an ordered set $X$ to the set $DX$ of all down-sets of $X$, ordered by inclusion, the unit $\yoneda_X:X\to DX$ maps a point $x$ to its principal down-set $\downc x$, and the multiplication $\yonmult_X:DDX\to DX$ is given by union $\calA\mapsto\bigcup\calA=\{x\in X\mid \yoneda_X(x)\in\calA\}$. We also note that a subset $A$ of $X$ is down-closed if and only if its characteristic map is monotone of type $X^\op\to\two$, hence $DX\simeq\two^{X^\op}$, and the function $Df:DX\to DY$ is the left adjoint of the inverse image function $DY\to DX,\,\varphi\mapsto\varphi\cdot f$. Let now $X$ be a topological space, then each filter $\frf$ on $\calO X$ defines the set $\{\frx\in UX \mid \frf\subseteq\frx\}$ of ultrafilters on $X$. In fact, these are precisely the closed subsets of $UX$ with respect to a certain topology $\tau$ on $UX$, and we write $X^\op$ for the topological space $(UX,\tau)$. Then $X^\op$ turns out to be exponentiable (=core-compact), and the filter space $FX$ is homeomorphic to the function space $\two^{X^\op}$. Under the identification $FX\simeq\two^{X^\op}$, both the unit and the multiplication of the Filter monad are formally similar to their counterparts in the down-set monad: $\yoneda_X(x)$ is the characteristic map of $\{\frx\in UX\mid \frx\to x\}$ and $\yonmult_X(\Psi)$ is the characteristic map of $\{\frx\in UX\mid U\yoneda_X(\frx)\in\Psi\}$. This description of the filter monad serves here just for explaining why we think of this monad as the ``down-set monad'' and will not be used further (except for Remark \ref{rem:sup_vs_cocompl}), therefore we refer for all the details to \citep[Example 4.10]{HT_LCls} and \citep[Subsection 2.5]{Hof_Cocompl} and to \citep{CH_Compl,Hof_DualityDistSp} for a description of the dual space $X^\op$. 

For any regular cardinal $\alpha$ or for $\alpha=\Omega$ the class of all cardinals, we say that a filter $\frf\in FX$ is \emph{unreachable by $\alpha$} (or simply: an $\alpha$-filter) whenever, for any family $(A_i)_{i\in I}$ of opens where $\#(I)<\alpha$, $\bigcup_{i\in I}A_i\in\frf$ implies $A_i\in\frf$ for some $i\in I$. The \emph{$\alpha$-filter monad} $\mFa{\alpha}=\fmonada{\alpha}$ on $\TOP$ is the submonad of $\mF$ defined by those filters unreachable by $\alpha$, here we implicitly use that $\yoneda_X(x)$ and $\yonmult_X(\frX)$ are $\alpha$-filters if $\frX\in FFX$ is so. Then $\mFa{0}=\mF$ is the filter monad, $\mFa{1}$ is the proper filter monad, $\mFa{\omega}$ is the prime filter monad and $\mFa{\Omega}$ the completely prime filter monad. 

We need to explain the analogies between $\alpha$-filters and certain down-sets mentioned in the table above. Every down-set $A\subseteq X$ of an ordered set $X$ induces a monotone map
\[
A\otimes-:\Up(X)\to\two,\,B\mapsto A\otimes B=\llbracket\exists x\in X\,.\,x\in A\And x\in B\rrbracket,
\]
where $\Up(X)$ denotes the set of all up-sets of $X$ ordered by inclusion. Similarly, every filter $\frf$ of opens of a topological space $X$ (which we identify with a closed subset $\calA$ of $UX$) induces a monotone map
\[
\calA\otimes-:\Clos(X)\to\two,\,B\mapsto\calA\otimes B=\llbracket\exists\frx\in UX\,.\,\frx\in\calA \And \frx\in UB\rrbracket,
\]
where $\Clos(X)$ denotes the set of all closed subsets of $X$ ordered by inclusion. Then $A$ is non-empty if and only if $A\otimes-$ preserves the top element of $\Up(X)$, and $A$ is directed if and only if $A\otimes-$ preserves finite infima. Similarly, $\frf$ is proper if and only if $\calA\otimes-$ preserves the top element of $\Clos(X)$, and $\frf$ is prime if and only if $\calA\otimes-$ preserves finite infima. In general, $\frf$ is an $\alpha$-filter if and only if $\calA\otimes-$ preserves $\alpha$-infima.

Throughout this paper we fix $\alpha\in\{0,1,\omega,\Omega\}$, and write $\mT=\monad$ for the monad $\mFa{\alpha}$ of filters unreachable by $\alpha$. As usual, $\TOP^{\mT}$ denotes the category of $\mT$-algebras and $\mT$-homomorphisms. A very important fact to know about our monad $\mT=\monad$ is that $\mT$ is of \emph{Kock-Z\"oberlein type} (see \citep{EF_SemDom}), which amounts to saying that $T\!\yoneda_X\le \yoneda_{TX}$, for every topological space $X$.  This property is very convenient since it ensures that a continuous map $l:TX\to X$ is the structure morphism of a $\mT$-algebra if and only if $l\cdot\yoneda_X=1_X$, and then $l\dashv\yoneda_X$ in $\TOP$. Therefore a topological space $X$ admits at most one $\mT$-algebra structure, and in this case we say that $X$ \emph{is} a $\mT$-algebra. Note that every $\mT$-algebra $X$ is T$_{_0}$ since $TX$ is so and $\yoneda_X:X\to TX$ is an embedding. From $l\dashv\yoneda_X$ one obtains
\[
 l(\frf)\le x\iff \frf\le\yoneda_X(x)\iff \calO(x)\subseteq\frf\iff\frf\to x
\]
for all $\frf\in TX$ and $x\in X$, hence $l$ sends a filter $\frf$ to its smallest limit point. This also shows that a continuous map $f:X\to Y$ between $\mT$-algebras is a $\mT$-homomorphism if and only if $f$ preserves smallest limit points of filters $\frf\in TX$. Thinking again in ordered terms, a $\mT$-algebra $X$ admits all suprema (=smallest limit points) of down-sets (=filters $\frf\in TX$), and a $\mT$-homomorphism preserves these suprema. As for ordered sets, one can easily show that a left adjoint continuous map $f:X\to Y$ between topological spaces preserves all suprema (=smallest limit points) which exist in $X$. However, we also point out a crucial difference. For a topological space $X$, in order to be a $\mF$-algebra it is \emph{not} enough that each filter $\frf\in FX$ has a smallest limit point $l(\frf)$. The problem here is that the map $\frf\mapsto l(\frf)$ might not be continuous (see \citep[Example 5.7]{HW_AppVCat} for an example). In other words, a topological space $X$ might have all suprema but is not yet cocomplete. 

The algebras for the these various kinds of filter monads are well-known and, for instance, described in \citep{EF_SemDom}. We recall here that the algebras for the 
\begin{itemize}
\item completely prime filter monad are precisely the sober spaces, see \citep{Joh_StoneSp}.
\item prime filter monad are precisely the stably (or well-) compact spaces, or, equivalently, the ordered compact Hausdorff spaces, see \citep{Nach_TopOrd,Sim82a,Wyl84a,Jung04}. 
\item proper filter monad are the continuous Scott domains \citep{Book_ContLat}.
\item filter monad are the continuous lattices, see \citep{Wyl85,Day_Filter}.
\end{itemize}
Also note that every $\mT$-algebra is sober since $\mFa{\Omega}\hrw\mFa{\omega}\hrw\mFa{1}\hrw\mFa{0}$.

\section{Cocomplete spaces}

In order to describe distributive spaces in the next section, it will be useful to know some properties of cocomplete spaces (= $\mT$-algebras), which is the topic of this section. Essentially we are going to show that \citeauthor{Sim82a}' description of the algebras for the prime filter monad as precisely the \emph{well-compact} spaces (see \citep{Sim82a}) is valid for the other monads too, we only need to exchange the way-below relation on $\calO X$ by a ``relative-to-$\mT$'' variant.

\begin{definition}
For a topological space $X$ and $U,V\subseteq X$ open, we say that $U$ is \emph{$\mT$-below} $V$, and write $U\tbelow V$, if every $\frf\in TX$ with $U\in\frf$ has a limit point in $V$.
\end{definition}
Clearly, if $\mT=\mFa{\omega}$ is the prime filter monad, then $\mT$-below specialises to the well-known way-below relation, and for $\alpha=\Omega$ one has $U\tbelow V$ if and only if  $U\subseteq V$. Furthermore, for $\mT=\mF$ being the filter monad, $U\tbelow V$ holds precisely when there exists some $x\in X$ with $U\subseteq\downc x\subseteq V$ (just consider $\frf=\upc\{U\}$ the principal filter induced by $U$ in the definition above).
\begin{lemma}
Let $X$ be a topological space. Then $X$ is a $\mT$-algebra if and only if
\begin{enumerate}
\item\label{Conda} $X$ is T$_{_0}$,
\item\label{Condc} every $\frf\in TX$ has a smallest limit point in $X$, and
\item\label{Condb} for every $x\in X$ and every $U\in\calO(x)$ there exists $V\in\calO(x)$ with $V\tbelow U$.
\end{enumerate}
\end{lemma}
\begin{proof}
Assume first that $X$ is a $\mT$-algebra with structure map $l:TX\to X$. We observed already above that $X$ satisfies conditions \eqref{Conda} and \eqref{Condc}. To see \eqref{Condb}, let $x\in X$ and put $\frF=e_{TX}(e_X(x))=Te_X(e_X(x))$. Then $Tl(\frF)=\calO(x)$, hence $l^{-1}(U)\in\frF$ for every $U\in\calO(x)$. But $l^{-1}(U)\in\frF$ if and only if there is some $V\in\calO(x)$ with $V^\#\subseteq l^{-1}(U)$, which is equivalent to $V\tbelow U$. Assume now that $X$ satisfies \eqref{Conda}, \eqref{Condc} and \eqref{Condb}. We define $l(\frf)$ to be the smallest limit point of $\frf\in TX$ in $X$, which is indeed unique thanks to \eqref{Conda}. To show that $l:TX\to X$ is continuous, let $\frf\in TX$ and $U\in\calO(l(\frf))$. Take $V\in\calO(l(\frf))$ with $V\tbelow U$. Then $\frf\in V^\#$ and $l(V^\#)\subseteq U$.
\end{proof}

\begin{remark}\label{rem:sup_vs_cocompl}
In the lemma above, \eqref{Condc} ensures that each filter $\frf\in FX$ has a smallest limit point, which is unique by \eqref{Conda}, and \eqref{Condb} guarantees that the map $FX\to X$ sending each filter $\frf$ to its smallest limit point is continuous. However, thought \eqref{Condc} does not provide us necessarily with a left adjoint of $\yoneda_X:X\to TX$ in $\TOP$, \eqref{Condc}  is sufficient for the existence of a left adjoint in $\ORD$. It is interesting to observe that the monotone map $\yoneda_X:X\to TX$ preserves all infima which exist in the ordered set $X$ provided that every prime filter of opens (or, equivalently, every ultrafilter) on the topological space $X$ has a smallest limit point. To see this, let $(x_i)_{i\in I}$ be a family of elements $x_i\in X$ with infimum $x\in X$. We wish to show that $\yoneda_X(x)=\bigwedge_{i\in I}\yoneda_X(x_i)$ in $FX$. Recall that the order on $FX$ is dual to inclusion, hence the infimum of a family of filters is the filter generated by their union. Here it is convenient to have $FX\simeq\two^{X^\op}$ since infima on the right hand side are calculated point-wise. Now one easily sees
\[
 \bigwedge_{i\in I}\yoneda_X(x_i)=\bigcap_{i\in I}\{\frx\in UX\mid \frx\to x_i\}
=\{\frx\in X\mid\forall\,i\in I\,.\,\frx\to x_i\}=\{\frx\in X\mid\frx\to x\}=\yoneda_X(x)
\]
since, if $\frx\in UX$ is an ultrafilter with smallest limit point $z$, then $\frx\to x_i$ for all $i\in I$ if and only if $z\le x_i$ for all $i\in I$ if and only if $z\le x$ if and only if $\frx\to x$. Hence, if additionally the underlying order of $X$ is complete, then every filter has a smallest limit point.
\end{remark}

\begin{definition}
Let $X$ be a topological space. Then $X$ is \emph{$\mT$-core-compact} if, for every $x\in X$ and every $U\in\calO(x)$, there exists $V\in\calO(x)$ with $V\tbelow U$. Furthermore, $X$ is \emph{$\mT$-stable} if, for finite families $(U_i)_i$ and $(V_i)_i$ with $V_i\tbelow U_i$ for every $i\in\{1,\ldots,n\}$ ($n\in\N$), one has $\bigcap_{i=1}^n V_i\tbelow\bigcap_{i=1}^n U_i$.
\end{definition}
Of course, in the definition of $\mT$-stable it is enough to consider the cases $n=0$ and $n=2$, and the first one amounts to saying that every $\frf\in TX$ converges in $X$. If $\mT=\mFa{\omega}$ is the prime filter monad, then $\mFa{\omega}$-core-compact means core-compact and $\mFa{\omega}$-stable means compact and stable in the sense of \citep{Sim82a}. A topological space $X$ is $\mF$-core-compact if, for every point $x\in X$ and every open neighbourhood $V$ of $x$, there exists a point $y\in V$ and a open neighbourhood $U$ of $x$ such that $U\subseteq\downc y$, that is, $X$ is a C-space (see \citep{Ern91a}). Both notions trivialise when $\mT=\mFa{\Omega}$ is the completely prime filter monad: every topological space is $\mFa{\Omega}$-core-compact and $\mFa{\Omega}$-stable.

The following result is essentially \citep[Lemma 3.7]{Sim82a}.
\begin{lemma}
Assume that $X$ is $\mT$-core-compact. Then $X$ is $\mT$-stable if and only if, for every $\frf\in TX$, the set $\lim\frf$ of limit points of $\frf$ is irreducible.
\end{lemma}
\begin{proof}
Assume first that $X$ is $\mT$-stable, and let $\frf\in TX$. Certainly, $\lim\frf\neq\varnothing$. Let $U_1,U_2\in\calO X$ with $U_1\cap\lim\frf\neq\varnothing$ and $U_2\cap\lim\frf\neq\varnothing$. Let $x_1\in U_1\cap\lim\frf$ and $x_2\in U_2\cap\lim\frf$, and choose $V_1\in\calO(x_1)$, $V_2\in\calO(x_2)$ with $V_1\tbelow U_1$ and $V_2\tbelow U_2$. Then $V_1\in\frf$ and $V_2\in\frf$, hence $V_1\cap V_2\in\frf$, and, since $V_1\cap V_2\tbelow U_1\cap U_2$, we conclude $U_1\cap U_2\cap\lim\frf\neq\varnothing$. Assume now that $\lim\frf$ is irreducible, for every $\frf\in TX$. In particular, $\lim\frf\neq\varnothing$, hence every $\frf\in TX$ converges. Assume now $V_1\tbelow U_1$ and $V_2\tbelow U_2$, and let $\frf\in TX$ with $V_1\cap V_2\in\frf$. Then $U_1\cap\lim\frf\neq\varnothing$ and $U_2\cap\lim\frf\neq\varnothing$, and consequently $U_1\cap U_2\cap\lim\frf\neq\varnothing$. 
\end{proof}

\begin{proposition}
Let $X$ be a topological space. Then $X$ is a $\mT$-algebra if and only if $X$ is sober, $\mT$-core-compact and $\mT$-stable. 
\end{proposition}

\section{Distributive spaces}

We turn now our attention to those $\mT$-algebras $X$ where the left adjoint $l:TX\to X$ of $\yoneda_X:X\to TX$ has a further left adjoint $t:X\to TX$ in $\TOP$, we call such a space \emph{$\mT$-distributive}. Certainly, $t\dashv l$ implies that $t$ is a $\mT$-homomorphism and, since $l:TX\to X$ is surjective, one has $l\cdot t=1_X$. Conversely, if we find a splitting $t:X\to TX$ of $l:TX\to X$ in $\TOP^{\mT}$, then necessarily $t\dashv l$ since $\mT$ is of Kock-Z\"oberlein type. In other words, a $\mT$-algebra $X$ is $\mT$-distributive precisely when its structure map $l:TX\to X$ is a split epimorphism in $\TOP^{\mT}$. It is also worth noting that a $\mT$-algebra admits such a splitting if and only if it is projective with respect to those $\mT$-homomorphisms which are split epimorphisms in $\TOP$. We write $\Spl{{\TOP^\mT}}$ for the full subcategory of $\TOP^{\mT}$ defined by the $\mT$-distributive spaces.

Throughout this section $X$ denotes a $\mT$-algebra with structure map $l:TX\to X$. For any $A\in\calO X$, we put
\[
 \mu(A):=\{x\in X\mid \text{$x=l(\frf)$ for some $\frf\in TX$ with $A\in\frf$}\}.
\]
In other words, we define a kind of \emph{closure} of $A$ where we join to $A$ all \emph{smallest} limit points of filters $\frf\in TX$ on $A$. Since open subsets $U,V\subseteq X$ are down-closed in the underlying order, one has $U\tbelow V$ if and only if $\mu(U)\subseteq V$. Also note that this closure becomes trivial for $\alpha=\Omega$, that is, $\mu(A)=A$.

\begin{lemma}
Let $X$ be a $\mT$-distributive space where $t:X\to TX$ denotes the left adjoint of $l:TX\to X$ in $\TOP$. Then
\[
 t(x)=\{A\in\calO X\mid x\in\mu(A)\}.
\]
\end{lemma}
\begin{proof}
By assumption, for any $\frf\in TX$ and any $x\in X$,
\[
 t(x)\le\frf\iff x\le l(\frf).
\]
Let $A\in\calO X$. If $x\in\mu(A)$, then $x=l(\frf)$ for some $\frf\in TX$ with $A\in\frf$; hence $A\in\frf\subseteq t(x)$. If $A\in t(x)$, then $x\in\mu(A)$ since $l(t(x))=x$.
\end{proof}

\begin{definition}
A $\mT$-algebra $X$ is called \emph{$\mT$-disconnected} if $\mu(A)$ is open, for every $A\in\calO X$.
\end{definition}

The designation ``$\mT$-disconnected'' is motivated by the notion of extremely disconnected space, which is a topological space where the closure of every open subset is open. Furthermore, a compact Hausdorff space $X$ (seen as a $\mFa{\omega}$-algebra) is $\mFa{\omega}$-disconnected if and only if $X$ is extremely disconnected since in this case $\mu$ is just the closure of $X$.

Note that every $\mT$-distributive space is $\mT$-disconnected since $\mu(A)=t^{-1}(A^\#)$. In the remainder of this section we will show that also the converse is true.

\begin{lemma}
Let $A,B\in\calO X$ and $(A_i)_{i\in I}$ be a family of open subsets $A_i$ of $X$ where $\#(I)<\alpha$. Then the following assertions hold.
\begin{enumerate}
\item\label{cond1} $A\subseteq\mu(A)$.
\item\label{cond2} If $A\subseteq B$, then $\mu(A)\subseteq\mu(B)$.
\item\label{cond2a} $\mu(\bigcup_{i\in I}A_i)\subseteq \bigcup_{i\in I}\mu(A_i)$
\item\label{cond3} $A\cap\mu(B)\subseteq\mu(A\cap B)$.
\item\label{cond4} If $X$ is $\mT$-disconnected, then $\mu\mu(A)\subseteq\mu(A)$.
\end{enumerate}
\end{lemma}
\begin{proof}
The assertions \eqref{cond1}, \eqref{cond2} and \eqref{cond2a} are obvious. To see \eqref{cond3}, let $x\in A\cap\mu(B)$. There is some $\frf\in TX$ with $B\in\frf$ and $x$ is a smallest limit point of $\frf$. Since $A$ is open, $A\in\frf$ and therefore $A\cap B\in\frf$. Assume now that $X$ is $\mT$-disconnected. The assertion is clear if $\mT$ is the completely prime filter monad. Note that in the other three cases $T:\TOP\to\TOP$ maps surjections to surjections. Let $x\in\mu\mu(A)$. Hence there is some $\frf\in TX$ with $l(\frf)=x$ and $\mu(A)\in\frf$. Since $l:TA\to\mu(A)$ is surjective, $Tl:TTA\to T\mu(A)$ is surjective as well, hence there is some $\frF\in TTA$ with $Tl(\frF)=\frf$. Hence $x=l(\yonmult_X(\frF)$ and $A\in\yonmult_X(\frF)$.
\end{proof}

\begin{corollary}
For all $A,B\in\calO X$, $\mu(A\cap B)=\mu(A)\cap\mu(B)$.
\end{corollary}
\begin{proof}
We calculate:
$\mu(A)\cap\mu(B)\subseteq\mu(A\cap\mu(B))\subseteq\mu(A\cap B)$.
\end{proof}

Therefore, if $X$ is $\mT$-disconnected, then
\[
 t(x):=\{A\in\calO X\mid x\in\mu(A)\}
\]
is a filter on $\calO X$ unreachable by $\alpha$, and we obtain a map $t:X\to TX$. Moreover, $t$ is continuous since $t^{-1}(A^\#)=\mu(A)$ for every $A\in\calO X$.

\begin{lemma}
If $X$ is $\mT$-disconnected, then $t:X\to TX$ is a $\mT$-homomorphism.
\end{lemma}
\begin{proof}
Let $\frf\in TX$ and $x\in X$ with $l(\frf)=x$. We wish to show that $t(x)=\yonmult_X(Tt(\frf)$, that is, for every $A\in\calO X$ one has $A\in t(x)$ if and only if $t^{-1}(A^\#)\in\frf$, which in turn is equivalent to $\mu(A)\in\frf$. Now, since $l(\frf)=x$, $x\in\mu(A)$ implies $\mu(A)\in\frf$, and $\mu(A)\in\frf$ implies $x\in\mu\mu(A)=\mu(A)$.
\end{proof}

\begin{lemma}
If $X$ is $\mT$-disconnected, then $l(t(x))=x$.
\end{lemma}
\begin{proof}
Clearly, $\calO(x)\subseteq t(x)$, hence $t(x)$ converges to $x$. Assume now that $t(x)\to y$. Let $U\in\calO(y)$. Since $X$ is $\mT$-core-compact, we find some $V\in\calO(y)$ with $\mu(V)\subseteq U$. From $t(x)\to y$ we get $V\in t(x)$, and therefore $x\in\mu(V)\subseteq U$. Hence $x\le y$.
\end{proof}
Summing up, we have shown:
\begin{theorem}\label{thm:Tdist}
A topological space $X$ is $\mT$-distributive if and only if $X$ is sober, $\mT$-core-compact, $\mT$-stable and $\mT$-disconnected.
\end{theorem}

\section{Dualities via idempotents split completion}

A further important result of order theory on which we wish to expand here states that the category of completely distributive lattices and suprema preserving maps is equivalent to the idempotents split completion of the category $\REL$ of sets and relations, and also to the idempotents split completion of the category $\MOD$ of ordered sets and modules (see \citep{RW_CCD4} for details). We do not need to search anymore for a topological variant of these equivalences since the authors of \citep{RW_CCD4} have already given a very general version of their result: in \citep{RW_Split} they show that, for a monad $\monadfont{D}$ on a category $\mathsf{C}$ where idempotents split,  the idempotents split completion $\kar(\catfont{C}_\monadfont{D})$ of the Kleisli category $\catfont{C}_\monadfont{D}$ of $\monadfont{D}$ is equivalent to the category $\Spl{{\catfont{C}^\monadfont{D}}}$ of split Eilenberg-Moore algebras (see below) for $\monadfont{D}$. The original result for completely distributive lattices one obtains for $\monadfont{D}$ being the down-set monad on $\ORD$ or the powerset monad on $\SET$. In \citep{Hof_DualityDistSp} we have already applied \citep{RW_Split} to various kinds of filter monads, and in this section we add a couple of further observations to this study.

A typical example of an idempotent $e:X\to X$ (i.e. $e=e\cdot e$) in a category $\catfont{A}$ is a morphism $e=s\cdot r$ where $r\cdot s=1$. One says that \emph{idempotents split} in $\catfont{A}$ whenever every idempotent is of this form. This is the case for every complete or cocomplete category, and a good example of a category where idempotents do not split is the category $\REL$ of relations (take $<:\R\relto\R$, for instance). For every category $\catfont{A}$ there exists its universal \emph{idempotents split completion} (which is equivalent to $A$ precisely when idempotents split in $\catfont{A}$), and it is an important fact for us that this completion is ``some sort of Cauchy completion for categories''. For instance, given a category $\catfont{A}$, its idempotents split completion can be calculated as follows:
\begin{itemize}
\item embed $\catfont{A}$ fully into a category $\catfont{B}$ where idempotents split,
\item then the idempotents split completion of $\catfont{A}$ is ``the closure of $\catfont{A}$ in $\catfont{B}$'', that is, the full subcategory of $\catfont{B}$ defined by the split subobjects in $\catfont{B}$ of objects of $\catfont{A}$.
\end{itemize}
Certainly, this procedure resembles the situation for metric spaces.

Adjunctions induce monads, and every monad is induced by an adjunction. In fact, given a monad $\mT=\monad$ on a category $\catfont{C}$, there are at least two adjunctions which induce $\mT$, namely the Kleisli construction $F_\mT\dashv G_\mT:\catfont{C}_{\mT}\leftrightarrows\catfont{C}$ and the Eilenberg-Moore construction $F^\mT\dashv G^\mT:\catfont{C}^{\mT}\leftrightarrows\catfont{C}$. The former is the smallest and the latter is the largest adjunction inducing $\mT$ in the sense that, whenever an adjunction $F\dashv G:\catfont{X}\leftrightarrows\catfont{C}$ induces $\mT$, then there are comparison functors $K:\catfont{C}_{\mT}\to\catfont{X}$ and $K':\catfont{C}^{\mT}\to\catfont{X}$ commuting with both the left and the right adjoints. It is useful here to observe here that $K:\catfont{C}_{\mT}\to\catfont{X}$ is always fully faithful.

In particular, there is a fully faithful comparison functor $\catfont{C}_{\mT}\to\catfont{C}^{\mT}$. Furthermore, idempotents split in $\catfont{C}^{\mT}$ provided they do so in $\catfont{C}$. Hence, by the above, the idempotents split completion of $\catfont{C}_{\mT}$ can be calculated as the closure of $\catfont{C}_{\mT}$ in $\catfont{C}^{\mT}$, which turns out to be the full subcategory $\Spl{{\catfont{C}^{\mT}}}$ of $\catfont{C}^{\mT}$ consisting of all those algebras $(X,l)$ where $l:TX\to X$ admits a homomorphic splitting, that its, there exists a $\mT$-homomorphism $t:X\to TX$ with $l\cdot t=1_X$. The point we wish to make here is that this is just one way to calculate the idempotents split completion of $\catfont{C}_{\mT}$, any other full embedding of $\catfont{C}_{\mT}$ into a category where idempotents split leads to a possibly different but necessarily equivalent completion. For instance, for any adjunction $F\dashv G:\catfont{X}^\op\leftrightarrows\catfont{C}$ which induces $\mT$ and where idempotents split in $\catfont{X}$ (and hence in $\catfont{X}^\op$), the comparison functor $K:\catfont{C}_{\mT}\to\catfont{X}^\op$ embeds $\catfont{C}_{\mT}$ fully in $\catfont{X}^\op$. Hence, the closure of $\catfont{C}_{\mT}$ in $\catfont{X}^\op$ is equivalent to $\Spl{{\catfont{C}^{\mT}}}$. We argue here that this procedure gives a unified and simple way to prove several duality theorems.

We return now to our monad $\mT=\mFa{\alpha}$ of $\alpha$-filters. By the discussion above, the category $\Spl{{\TOP^\mT}}$ of $\mT$-distributive spaces and $\mT$-homomorphisms is equivalent to the idempotents split completion of the Kleisli category $\TOP_{\mT}$ of $\mT$. But the monad $\mT$ is also induced by the adjunction
\[
\xymatrix{\SLat{\alpha}^\op\stackrel{\eta}{\Rw}\ar@<0.5ex>[rr]^{F_{\!\!\alpha}}&&
\stackrel{\eps}{\Lw}\TOP,\ar@<1ex>[ll]^{\calO}}
\]
where $\SLat{\alpha}$ denotes the category with objects all meet semi-lattices which admit all suprema of size smaller then $\alpha$ ($\alpha$-suprema for short) and where finite infima commute with $\alpha$-suprema, and the morphisms are the monotone maps preserving finite infima and $\alpha$-suprema. Note that $\SLat{\alpha}$ is an algebraic category, hence idempotents split. The comparison functor $K:\TOP_\mT\to\SLat{\alpha}^\op$ sends a space $X$ to $\calO X$, and $r:X\to TY$ to $\calO Y\to\calO X,\,B\mapsto r^{-1}(B^\#)$.

\begin{lemma}
An object $L$ in $\SLat{\alpha}^\op$ belongs to the closure of $K$ if and only if $L$ is a frame where $\alpha$-filters separate points.
\end{lemma}
\begin{proof}
Assume first that $S$ is in the closure of $K$, hence there are morphisms $s:S\to\calO X$ and $r:\calO X\to S$ with $rs=1_S$, for some topological space $X$. The assertion follows now from the following obvious facts about $s:M\to L$, $r:L\to M$ in $\ORD$ with $rs=1_M$.\\
\textsc{Fact 1:} If $L$ is complete, then so is $M$. Moreover, for a family $(x_i)_{i\in I}$ of elements of $M$, the supremum of $(x_i)_{i\in I}$ in $M$ is given by $r(\Sup_{i\in I}s(x_i))$.\\
\textsc{Fact 2:} If $r$ and $s$ preserve finite infima and $L$ is a frame, then so is $M$.\\
\textsc{Fact 3:} If $s$ preserves $\alpha$-suprema and $\alpha$-filters separate points in $L$, then they do so in $M$.

To see the other direction, for any such frame we can consider the unit
\[
 \eta_L:L\to\calO(F_{_{\! \alpha}}L),\,x\mapsto x^\#
\]
of the adjunction above which is an embedding in $\SLat{\alpha}$ by hypothesis. Furthermore, $\eta_L$ preserves all infima and its left adjoint preserves finite infima (which can be shown as in \citep[Proposition 7.3]{Hof_DualityDistSp}, for instance).
\end{proof}
The full subcategory of $\SLat{\alpha}$ determined by those frames where $\alpha$-filters separate points we denote as $\Frm{\alpha}$. Note that $\Frm{\alpha}$ contains in fact all frames for $\alpha=\omega,1,0$, and for $\alpha=\Omega$ it is the category of spatial frames and frame homomorphisms. The discussion above implies now

\begin{theorem}\label{thm:duality}
For any $\alpha\in\{0,1,\omega,\Omega\}$, $\Frm{\alpha}^\op\simeq\Spl{\TOP^{\mFa{\alpha}}}$.
\end{theorem}

\begin{corollary}
For any $\alpha\in\{0,1,\omega,\Omega\}$, a topological space $X$ is $\mFa{\alpha}$-distributive if and only if $X$ is the $\alpha$-filter space of a frame.
\end{corollary}

\begin{examples}
For $\alpha=\Omega$, Theorem \ref{thm:duality} just confirms the well-known duality between spatial frames and sober spaces.\\
For $\alpha=\omega$, the category $\SLat{\omega}$ is precisely the category $\DLAT$ of bounded distributive lattices and lattice homomorphism, and $\Frm{\omega}$ its full subcategory consisting of all frames. On the other side, $\TOP^{\mFa{\omega}}$ is the category of all stably compact spaces or, equivalently, ordered compact Hausdorff spaces. Every $\mFa{\omega}$-distributive space is a Priestley space, in fact, $\mFa{\omega}$-distributive spaces are exactly the f-spaces in the sense of \citep{PS_Frames_vs_Priestley}. Hence, Theorem \ref{thm:duality} states the restriction of Priestley's duality theorem (see \citep{Pri70,Pri72}) to frames and f-spaces respectively.\\
For $\alpha=0$, $\SLat{0}$ is the category $\SLAT$ of meet semi-lattices and meet semi-lattice homomorphisms, and $\Frm{0}$ is the category $\FRM_\wedge$ of frames and meet semi-lattice homomorphisms. From Theorem \ref{thm:duality} we infer that $\FRM_\wedge$ is dually equivalent to $\Spl{\TOP^\mF}$.
\end{examples}

\def\cprime{$'$}

\end{document}